\documentclass[11pt]{article}

\usepackage{amssymb}
\usepackage{amsmath,amsthm}
\usepackage{verbatim}
\usepackage{graphicx}
\newtheorem{definition}{Definition}
\newtheorem{theorem}{Theorem}
\newtheorem{lemma}[theorem]{Lemma}

\begin{document}


\title{Stability for \(t\)-intersecting families of permutations}

\author{David Ellis}
\maketitle

\begin{abstract}
A family of permutations \(\mathcal{A} \subset S_{n}\) is said to be \(t\)-\textit{intersecting} if any two permutations in \(\mathcal{A}\) agree on at least \(t\) points, i.e. for any \(\sigma, \pi \in \mathcal{A}\), \(|\{i \in [n]: \sigma(i)=\pi(i)\}| \geq t\). It was proved by Friedgut, Pilpel and the author in \cite{friedgutpilpel} that for \(n\) sufficiently large depending on \(t\), a \(t\)-intersecting family \(\mathcal{A} \subset S_{n}\) has size at most \((n-t)!\), with equality only if \(\mathcal{A}\) is a coset of the stabilizer of \(t\) points (or `\(t\)-coset' for short), proving a conjecture of Deza and Frankl. Here, we first obtain a rough stability result for \(t\)-intersecting families of permutations, namely that for any \(t \in \mathbb{N}\) and any positive constant \(c\), if \(\mathcal{A} \subset S_{n}\) is a \(t\)-intersecting family of permutations of size at least \(c(n-t)!\), then there exists a \(t\)-coset containing all but at most a \(O(1/n)\)-fraction of \(\mathcal{A}\). We use this to prove an exact stability result: for \(n\) sufficiently large depending on \(t\), if \(\mathcal{A} \subset S_{n}\) is a \(t\)-intersecting family which is not contained within a \(t\)-coset, then \(\mathcal{A}\) is at most as large as the family
\begin{eqnarray*}
\mathcal{D} & = & \{\sigma \in S_{n}:\ \sigma(i)=i\ \forall i \leq t,\ \sigma(j)=j\ \textrm{for some}\ j > t+1\}\\
&& \cup \{(1\ t+1),(2\ t+1),\ldots,(t \ t+1)\}
\end{eqnarray*}
which has size \((1-1/e+o(1))(n-t)!\). Moreover, if \(\mathcal{A}\) is the same size as \(\mathcal{D}\) then it must be a `double translate' of \(\mathcal{D}\), meaning that there exist \(\pi,\tau \in S_{n}\) such that \(\mathcal{A}=\pi \mathcal{D} \tau\). The \(t=1\) case of this was a conjecture of Cameron and Ku and was proved by the author in \cite{cameronkupaper}. We build on the methods of \cite{cameronkupaper}, but the representation theory of \(S_{n}\) and the combinatorial arguments are more involved. We also obtain an analogous result for \(t\)-intersecting families in the alternating group \(A_{n}\).
\end{abstract}

\section{Introduction}
We work first on the symmetric group \(S_{n}\), the group of all permutations of \(\{1,2,\ldots,n\} = [n]\). A family of permutations \(\mathcal{A} \subset S_{n}\) is said to be \(t\)-\textit{intersecting} if any two permutations in \(\mathcal{A}\) agree on at least \(t\) points, i.e. for any \(\sigma, \pi \in \mathcal{A}\), \(|\{i \in [n]: \sigma(i)=\pi(i)\}| \geq t\). Deza and Frankl \cite{dezafrankl} conjectured that for \(n\) sufficiently large depending on \(t\), a \(t\)-intersecting family \(\mathcal{A} \subset S_{n}\) has size at most \((n-t)!\); this became known as the Deza-Frankl conjecture. It was proved in 2008 by Friedgut, Pilpel and the author in \cite{friedgutpilpel} using eigenvalue techniques and the representation theory of the symmetric group; it was also shown in \cite{friedgutpilpel} that equality holds only if \(\mathcal{A}\) is a coset of the stabilizer of \(t\) points (or `\(t\)-coset' for short). In this paper, we will first prove a rough stability result for \(t\)-intersecting families of permutations. Namely, we show that for any fixed \(t \in \mathbb{N}\) and  \(c>0\), if \(\mathcal{A} \subset S_{n}\) is a \(t\)-intersecting family of size at least \(c (n-t)!\), then there exists a \(t\)-coset \(\mathcal{C}\) such that \(|\mathcal{A} \setminus \mathcal{C}| \leq \Theta((n-t-1)!)\), i.e. \(\mathcal{C}\) contains all but at most a \(O(1/n)\)-fraction of \(\mathcal{A}\).

We then use some additional combinatorial arguments to prove an exact stability result: for \(n\) sufficiently large depending on \(t\), if \(\mathcal{A} \subset S_{n}\) is a \(t\)-intersecting family which is not contained within a \(t\)-coset, then \(\mathcal{A}\) is at most as large as the family
\begin{eqnarray*}
\mathcal{D} & = & \{\sigma \in S_{n}:\ \sigma(i)=i\ \forall i \leq t,\ \sigma(j)=j\ \textrm{for some}\ j > t+1\}\\
&& \cup \{(1\ t+1),(2\ t+1),\ldots,(t \ t+1)\}
\end{eqnarray*}
which has size \((1-1/e+o(1))(n-t)!\). Moreover, if \(\mathcal{A}\) is the same size as \(\mathcal{D}\), then it must be a `double translate' of \(\mathcal{D}\), meaning that there exist \(\pi,\tau \in S_{n}\) such that \(\mathcal{A}=\pi \mathcal{D} \tau\). Note that if \(\mathcal{F} \subset S_{n}\), any double translate of \(\mathcal{F}\) has the same size as \(\mathcal{F}\), is \(t\)-intersecting iff \(\mathcal{F}\) is and is contained within a \(t\)-coset of \(S_{n}\) iff \(\mathcal{F}\) is; this will be our notion of `isomorphism'.

In other words, if we demand that our \(t\)-intersecting family \(\mathcal{A} \subset S_{n}\) is not contained within a \(t\)-coset of \(S_{n}\), then it is best to take \(\mathcal{A}\) such that all but \(t\) of its permutations are contained within some \(t\)-coset.

One may compare this with the situation for \(t\)-intersecting families of \(r\)-sets. We say a family \(\mathcal{A} \subset [n]^{(r)}\) of \(r\)-element subsets of \([n]\) is
\(t\)-\textit{intersecting} if any two of its sets contain at least \(t\) elements in common, i.e. \(|x \cap y| \geq t\) for any \(x,y \in \mathcal{A}\). Wilson \cite{wilson} proved using an eigenvalue technique that provided \(n \geq (t+1)(r-t+1)\), a \(t\)-intersecting family \(\mathcal{A} \subset [n]^{(r)}\) has size at most \({n-t \choose r-t}\), and that for \(n > (t+1)(r-t+1)\), equality holds only if \(\mathcal{A}\) consists of all \(r\)-sets containing some fixed \(t\)-set. Later, Ahlswede and Khachatrian \cite{ahslwedekhachatrian} characterized the \(t\)-intersecting families of maximum size in \([n]^{(r)}\) for all values of \(t,r\) and \(n\) using entirely combinatorial methods based on left-compression. They also proved that for \(n > (t+1)(r-t+1)\), if \(\mathcal{A} \subset [n]^{(r)}\) is \(t\)-intersecting and \textit{non-trivial}, meaning that there is no \(t\)-set contained in all of its members, then \(\mathcal{A}\) is at most as large as the family
\[\{x \in [n]^{(r)}:\ [t] \subset x,\ x \cap \{t+1,\ldots,r+1\} \neq \emptyset\} \cup \{[r+1] \setminus \{i\}:\ i \in [t]\}\]
if \(r > 2t+1\), and at most as large as the family  
\[\{x \in [n]^{(r)}: |x \cap [t+2]| \geq t+1\}\]
if \(r \leq 2t+1\). This had been proved under the assumption \(n \geq n_{1}(r,t)\) by Frankl \cite{frankl} in 1978. Note that the first family above is `almost trivial', and is the natural analogue of our family \(\mathcal{D}\).

The \(t=1\) case of our result was a conjecture of Cameron and Ku and was proved by the author in \cite{cameronkupaper}. We build on the methods of \cite{cameronkupaper}, but the representation theory of \(S_{n}\) and the combinatorial arguments required are more involved.

We also obtain analogous results for \(t\)-intersecting families of permutations in the alternating group \(A_{n}\). We use the methods of \cite{friedgutpilpel} to show that for \(n\) sufficiently large depending on \(t\), if \(\mathcal{A} \subset A_{n}\) is \(t\)-intersecting, then \(|\mathcal{A}|\leq (n-t)!/2\). Interestingly, it does not seem possible to use the methods of \cite{friedgutpilpel} to show that equality holds only if \(\mathcal{A}\) is a coset of the stabilizer of \(t\) points. Instead, we deduce this from a stability result. Using the same techniques as for \(S_{n}\), we prove that if \(\mathcal{A} \subset A_{n}\) is \(t\)-intersecting but not contained within a \(t\)-coset, then it is at most as large as the family
\begin{eqnarray*}
\mathcal{E}& = &\{\sigma \in A_{n}: \sigma(i)=i\ \forall i \leq t,\ \sigma(j)=(n-1\ n)(j)\ \textrm{for some}\ j > t+1\} \\
&&\cup \{(1\ t+1)(n-1\ n),(2\ t+1)(n-1\ n),\ldots,(t \ t+1)(n-1\ n)\}
\end{eqnarray*}
which has size \((1-1/e+o(1))(n-t)!/2\); if \(\mathcal{A}\) is the same size as \(\mathcal{E}\), then it must be a double translate of \(\mathcal{E}\), meaning that \(\mathcal{A}=\pi \mathcal{E} \tau\) for some \(\pi,\tau \in A_{n}\).

\section{Background}
In \cite{friedgutpilpel}, in order to prove the Deza-Frankl conjecture, we constructed (for \(n\) sufficiently large depending on \(t\)) a weighted graph \(Y\) which was a real linear combination of Cayley graphs on \(S_{n}\) generated by conjugacy-classes of permutations with less than \(t\) fixed points, such that the matrix \(A\) of weights of \(Y\) had maximum eigenvalue \(1\) and minimum eigenvalue
\[\omega_{n,t} = -\frac{1}{n(n-1)\ldots(n-t+1)-1}\]
The \(1\)-eigenspace was the subspace of \(\mathbb{C}[S_{n}]\) consisting of the constant functions. The direct sum of the 1-eigenspace and the \(\omega_{n,t}\)-eigenspace was the subspace \(V_{t}\) of \(\mathbb{C}[S_{n}]\) spanned by the characteristic vectors of the \(t\)-cosets of \(S_{n}\). All other eigenvalues were \(O(|\omega_{n,t}|/n^{1/6})\); this can in fact be improved to \(O(|\omega_{n,t}|/n)\), but any bound of the form \(o(|\omega_{n,t}|)\) will suffice for our purposes. We then appealed to a weighted version of Hoffman's bound (Theorem 11 in \cite{friedgutpilpel}):

\begin{theorem}
\label{thm:weightedhoffman}
Let \(A\) be a real, symmetric, \(N \times N\) matrix with real eigenvalues \(\lambda_{1} \geq \lambda_{2} \geq \ldots \geq \lambda_{N}\) (where \(\lambda_{1} > 0\)), such that the all-1's vector \(\mathbf{f}\) is an eigenvector of \(A\) with eigenvalue \(\lambda_{1}\), i.e. all row and column sums of \(A\) equal \(\lambda_{1}\). Let \(X \subset [N]\) such that \(A_{x,y} = 0\) for any \(x,y \in X\). Let \(U\) be the direct sum of the subspace of constant vectors and the \(\lambda_{N}\)-eigenspace. Then
\[|X| \leq \frac{|\lambda_{N}|}{\lambda_{1}+|\lambda_{N}|}N\]
and equality holds only if the characteristic vector \(v_{X}\) lies in the subspace \(U\).
\end{theorem}
Applying this to our weighted graph \(Y\) proved the Deza-Frankl conjecture:
\begin{theorem}
\label{thm:dezafranklconjecture}
For \(n\) sufficiently large depending on \(t\), a \(t\)-intersecting family \(\mathcal{A} \subset S_{n}\) has size \(|\mathcal{A}| \leq (n-t)!\).
\end{theorem}

Note that equality holds only if the characteristic vector \(v_{\mathcal{A}}\) of \(\mathcal{A}\) lies in the subspace \(V_{t}\) spanned by the characteristic vectors of the \(t\)-cosets of \(S_{n}\). It was proved in \cite{friedgutpilpel} that the Boolean functions in \(V_{t}\) are precisely the disjoint unions of \(t\)-cosets of \(S_{n}\), implying that equality holds only if \(\mathcal{A}\) is a \(t\)-coset of \(S_{n}\).

We also appealed to the following cross-independent weighted version of Hoffman's bound:

\begin{theorem}
\label{thm:weightedcrosshoffman}
Let \(A\) be as in Theorem \ref{thm:weightedhoffman}, and let \(\nu = \max(|\lambda_{2}|,|\lambda_{N}|)\). Let \(X, Y \subset [N]\) such that \(A_{x,y} = 0\) for any \(x \in X\) and \(y \in Y\). Let \(U\) be the direct sum of the subspace of constant vectors and the \(\pm \nu\)-eigenspaces. Then
\[|X||Y| \leq \left(\frac{\nu}{\lambda_{1}+\nu}N\right)^{2}\]
and equality holds only if \(|X|=|Y|\) and the characteristic vectors \(v_{X}\) and \(v_{Y}\) lie in the subspace \(U\).
\end{theorem}

Applying this to our weighted graph \(Y\) yielded:

\begin{theorem}
\label{thm:tcrossintersecting}
For \(n\) sufficiently large depending on \(t\), if \(\mathcal{A},\mathcal{B} \subset S_{n}\) are \(t\)-cross-intersecting, then \(|\mathcal{A}||\mathcal{B}| \leq ((n-t)!)^{2}\).
\end{theorem}

This will be a crucial tool in our stability analysis. Note that if equality holds in Theorem \ref{thm:tcrossintersecting}, then the characteristic vectors \(v_{\mathcal{A}}\) and \(v_{\mathcal{B}}\) lie in the subspace \(V_{t}\) spanned by the characteristic vectors of the \(t\)-cosets of \(S_{n}\), so by the same argument as before, \(\mathcal{A}\) and \(\mathcal{B}\) must both be equal to the same \(t\)-coset of \(S_{n}\).

We will need the following `stability' version of Theorem \ref{thm:weightedhoffman}:

\begin{lemma}
\label{lemma:stabilityhoffman}
Let \(A\), \(X\) and \(U\) be as in Theorem \ref{thm:weightedhoffman}. Let \(\alpha = |X|/N\). Let \(\lambda_{M}\) be the negative eigenvalue of second largest modulus. Equip \(\mathbb{C}^{N}\) with the inner product:
\[\langle x,y \rangle = \frac{1}{N}\sum_{i=1}^{N} \bar{x_{i}}y_{i}\]
and let
\[||x|| = \sqrt{\frac{1}{N}\sum_{i=1}^{N}|x_{i}|^{2}}\]
be the induced norm. Let \(D\) be the Euclidean distance from the characteristic vector \(v_{X}\) of \(X\) to the subspace \(U\), i.e. the norm \(||P_{U^{\perp}}(v_{X})||\) of the projection of \(v_{X}\) onto \(U^{\perp}\). Then
\[D^{2} \leq \frac{(1-\alpha)|\lambda_{N}| - \lambda_{1} \alpha}{|\lambda_{N}|-|\lambda_{M}|}\alpha \]
\end{lemma}
For completeness, we include a proof:
\begin{proof}
Let \(u_{1} = \mathbf{f}, u_{2},\ldots,u_{N}\) be an orthonormal basis of real eigenvectors of \(A\) corresponding to the eigenvalues \(\lambda_{1},\ldots,\lambda_{N}\). Write
\[v_{X}=\sum_{i=1}^{N} \xi_{i} u_{i}\]
as a linear combination of the eigenvectors of \(A\); we have \(\xi_{1}=\alpha\) and
\[\sum_{i=1}^{N} \xi_{i}^{2} = ||v_{X}||^{2} = |X|/N = \alpha\]
Then we have the crucial property:
\[0=\sum_{x,y \in X}A_{x,y}=v_{X}^{\top} A v_{X} = \sum_{i=1}^{N} \lambda_{i} \xi_{i}^{2} \geq \lambda_{1}\xi_{1}^{2} + \lambda_{N} \sum_{i:\lambda_{i}=\lambda_{N}} \xi_{i}^{2} + \lambda_{M} \sum_{i>1:\lambda_{i} \neq \lambda_{N}}\xi_{i}^{2}\]
Note that
\[\sum_{i>1:\lambda_{i} \neq \lambda_{N}}\xi_{i}^{2} = D^{2}\]
and
\[\sum_{i:\lambda_{i}=\lambda_{N}} \xi_{i}^{2} = \alpha - \alpha^{2} - D^{2}\] 
so we have
\[0 \geq \lambda_{1} \alpha^{2} + \lambda_{N} (\alpha -\alpha^{2} -D^{2}) + \lambda_{M} D^{2}\]
Rearranging, we obtain:
\[D^{2} \leq  \frac{(1-\alpha)|\lambda_{N}| - \lambda_{1} \alpha}{|\lambda_{N}|-|\lambda_{M}|}\alpha\]
as required.
\end{proof}
Our weighted graph \(Y\) has \(\lambda_{N} = \omega_{n,t}\) and \(|\lambda_{M}| = O(|\omega_{n,t}|/n^{1/6})\), so applying the above result to a \(t\)-intersecting family \(\mathcal{A} \subset S_{n}\) gives:
\begin{equation}
\label{eq:smallprojection}
||P_{V_{t}^{\perp}}(v_{\mathcal{A}})||^{2} \leq (1-|\mathcal{A}|/(n-t)!)(1+O(n^{1/6}))|\mathcal{A}|/n!
\end{equation}
Next, we find a formula for the projection \(P_{V_{t}}(v_{\mathcal{A}})\) of the characteristic vector of \(\mathcal{A}\) onto the subspace \(V_{t}\) spanned by the characteristic vectors of the \(t\)-cosets of \(S_{n}\). But first, we need some background on non-Abelian Fourier analysis and the representation theory of the symmetric group.

\subsection*{Background from non-Abelian Fourier analysis}

We now recall some information we need from \cite{friedgutpilpel}. [Notes for algebraists are included in square brackets and may be ignored without prejudicing the reader's understanding.]

If \(G\) is a finite group, a \textit{representation} of \(G\) is a vector space \(W\) together with a group homomorphism \(\rho:G \to \textrm{GL}(W)\) from \(G\) to the group of all automorphisms of \(W\), or equivalently a linear action of \(G\) on \(W\). If \(W = \mathbb{C}^{m}\), then \(\textrm{GL}(W)\) can be identified with the group of all complex invertible \(m \times m\) matrices; we call \(\rho\) a \textit{complex matrix representation} of degree (or dimension) \(m\). [Note that \(\rho\) makes \(\mathbb{C}^{m}\) into a \(\mathbb{C}G\)-module of dimension \(m\).]

We say a representation \((\rho,W)\) is \textit{irreducible} if it has no proper subrepresentation, i.e. no proper subspace of \(W\) is fixed by \(\rho(g)\) for every \(g \in G\). We say that two (complex) representations \((\rho,W)\) and \((\rho',W')\) are \textit{equivalent} if there exists a linear isomorphism \(\phi:W \to W'\) such that \(\rho'(g) \circ \phi = \phi \circ \rho(g)\ \forall g \in G\).

For any finite group \(G\), there are only finitely many equivalence classes of irreducible complex representations of \(G\). Let \((\rho_{1},\rho_{2},\ldots,\rho_{k})\) be a complete set of pairwise non-equivalent complex irreducible matrix representations of \(G\) (i.e. containing one from each equivalence class of complex irreducible representations).

\begin{definition}
The {\em (non-Abelian) Fourier transform} of a function \(f: G \to \mathbb{C}\) at the irreducible representation \(\rho_{i}\) is the matrix
\[\hat{f}(\rho_{i}) = \frac{1}{|G|} \sum_{g \in G} f(g) \rho_{i}(g)\]
\end{definition}

Let \(V_{\rho_{i}}\) be the subspace of functions whose Fourier transform is concentrated on \(\rho_{i}\), i.e. with \(\hat{f}(\rho_{j}) = 0\) for each \(j \neq i\). [Identifying the space \(\mathbb{C}[G]\) of all complex-valued functions on \(G\) with the \textit{group module} \(\mathbb{C}G\), \(V_{\rho_{i}}\) is the sum of all submodules of the group module isomorphic to the module defined by \(\rho_{i}\); it has dimension \(\dim(V_{\rho_{i}}) = (\dim(\rho_{i}))^{2}\). The group module decomposes as
\[\mathbb{C}G = \bigoplus_{i=1}^{k} V_{\rho_{i}}\]
Write \(\textrm{Id} = \sum_{i=1}^{k}e_{i}\), where \(e_{i} \in V_{\rho_{i}}\) for each \(i \in [k]\). The \(e_{i}\)'s are called the \textit{primitive central idempotents} of \(\mathbb{C}G\); they are given by the following formula:
\[e_{i} = \frac{\dim(\rho_{i})}{|G|} \sum_{g \in G} \chi_{i}(g^{-1}) g\]
They are in the \textit{centre} \(Z(\mathbb{C}G)\) of the group module, and satisfy \(e_{i}e_{j} = \delta_{i,j}\). Note that \(V_{\rho_{i}}\) is the two-sided ideal of \(\mathbb{C}G\) generated by \(e_{i}\). For any \(x \in \mathbb{C}G\), the unique decomposition of \(x\) into elements of the \(V_{\rho_{i}}\)'s is given by \(x = \sum_{i=1}^{k} e_{i}x\).]

A function \(f: G\to \mathbb{C}\) may be recovered from its Fourier transform using the Fourier Inversion Formula:
\[f(g) = \sum_{i=1}^{k} \dim(\rho_{i}) \textrm{Tr}\left(\hat{f}(\rho_{i})\rho_{i}(g^{-1})\right)\]
where \(\textrm{Tr}(M)\) denotes the trace of the matrix \(M\). It follows from this that the projection of \(f\) onto \(V_{\rho_{i}}\) has \(g\)-coordinate
\[P_{V_{\rho_{i}}}(f)_{g} = \frac{\dim(\rho_{i})}{|G|}\sum_{h \in G} f(h)\textrm{Tr}(\rho_{i}(h g^{-1})) = \frac{\dim(\rho_{i})}{|G|}\sum_{h \in G} f(h)\chi_{\rho_{i}}(h g^{-1})\]
where \(\chi_{\rho_{i}}(g) = \textrm{Tr}(\rho_{i}(g))\) denotes the character of the representation \(\rho_{i}\). 

\subsection*{Background on the representation theory of \(S_{n}\)}
A \emph{partition} of \(n\) is a non-increasing sequence of positive integers summing to \(n\), i.e. a sequence $\alpha = (\alpha_1, \ldots, \alpha_k)$ with \(\alpha_{1} \geq \alpha_{2} \geq \ldots \geq \alpha_{l} \geq 1\) and \(\sum_{i=1}^{l} \alpha_{i}=n\); we write \(\alpha \vdash n\). For example, \((3,2,2) \vdash 7\); we sometimes use the shorthand \((3,2,2) = (3,2^{2})\).

The \emph{cycle-type} of a permutation \(\sigma \in S_{n}\) is the partition of \(n\) obtained by expressing \(\sigma\) as a product of disjoint cycles and listing its cycle-lengths in non-increasing order. The conjugacy-classes of \(S_{n}\) are precisely
\[\{\sigma \in S_{n}: \textrm{ cycle-type}(\sigma) = \alpha\}_{\alpha \vdash n}.\]
Moreover, there is an explicit 1-1 correspondence between irreducible representations of \(S_{n}\) (up to isomorphism) and partitions of \(n\), which we now describe.

Let \(\alpha = (\alpha_{1}, \ldots, \alpha_{l})\) be a partiton of \(n\). The \emph{Young diagram} of $\alpha$ is
  an array of $n$ dots, or cells, having $l$ left-justified rows where row $i$
  contains $\alpha_i$ dots. For example, the Young diagram of the partition \((3,2^{2})\) is\\
  \\
  \begin{tabular}{ccc} \(\bullet\) & \(\bullet\) & \(\bullet\)\\
\(\bullet\) & \(\bullet\) &\\
\(\bullet\) & \(\bullet\) &\\
\end{tabular}\\
\\
\\
If the array contains the numbers \(\{1,2,\ldots,n\}\) in some order in place of the dots, we call it an \(\alpha\)-\emph{tableau}; for example,\\
  \\
  \begin{tabular}{ccc} 6 & 1 & 7\\
5 & 4 &\\
3 & 2 &\\
\end{tabular}\\
\\
is a \((3,2^{2})\)-tableau. Two \(\alpha\)-tableaux are said to be \emph{row-equivalent} if for each row, they have the same numbers in that row. If an \(\alpha\)-tableau \(s\) has rows \(R_{1},\ldots,R_{l} \subset [n]\) and columns \(C_{1},\ldots,C_{k} \subset [n]\), we let \(R_{s} = S_{R_{1}} \times S_{R_{2}} \times \ldots \times S_{R_{l}}\) be the row-stablizer of \(s\) and \(C_s = S_{C_{1}} \times S_{C_{2}} \times \ldots \times S_{C_{k}}\) be the column-stabilizer.

An \(\alpha\)-\emph{tabloid} is an \(\alpha\)-tableau with unordered row entries (or formally, a row-equivalence class of \(\alpha\)-tableaux); given a tableau \(s\), we write \([s]\) for the tabloid it produces. For example, the \((3,2^{2})\)-tableau above produces the following \((3,2^{2})\)-tabloid\\
\\
\begin{tabular}{ccc} \{1 & 6 & 7\}\\
\{4 &\ 5\} &\\
\{2 &\ 3\} &\\
\end{tabular}\\
\\
\\
Consider the natural left action of \(S_{n}\) on the set \(X^{\alpha}\) of all \(\alpha\)-tabloids; let \(M^{\alpha} = \mathbb{C}[X^{\alpha}]\) be the corresponding permutation module, i.e. the complex vector space with basis \(X^{\alpha}\) and \(S_{n}\) action given by extending this action linearly. Given an \(\alpha\)-tableau \(s\), we define the corresponding \(\alpha\)-\textit{polytabloid}
\[e_{s} := \sum_{\pi \in C_{s}} \epsilon(\pi) \pi[s]\]
We define the \textit{Specht module} \(S^{\alpha}\) to be the submodule of \(M^{\alpha}\) spanned by the \(\alpha\)-polytabloids:
\[S^{\alpha} = \textrm{Span}\{e_s:\ s \textrm{ is an } \alpha\textrm{-tableau}\}.\]
A central observation in the representation theory of \(S_{n}\) is that \textit{the Specht modules are a complete set of pairwise non-isomorphic, irreducible representations of} $S_n$. Hence, any irreducible representation \(\rho\) of \(S_{n}\) is isomorphic to some \(S^{\alpha}\). For example, \(S^{(n)} = M^{(n)}\) is the trivial representation; \(M^{(1^{n})}\) is the left-regular representation, and \(S^{(1^{n})}\) is the sign representation \(S\).

We say that a tableau is \textit{standard} if the numbers strictly increase along each row and down each column. It turns out that for any partition \(\alpha\) of \(n\),
\[\{e_{t}: t \textrm{ is a standard }\alpha \textrm{-tableau}\}\]
is a basis for the Specht module \(S^{\alpha}\).

Given a partition \(\alpha\) of \(n\), for each cell \((i,j)\) in its Young diagram, we define the `hook-length' \((h_{i,j}^{\alpha})\) to be the number of cells in its `hook' (the set of cells in the same row to the right of it or in the same column below it, including itself) --- for example, the hook-lengths of \((3,2^{2})\) are as follows:\\

\begin{tabular}{ccc}
5 & 4 & 1 \\
3 & 2 &\\
2 & 1 &
\end{tabular}\\
\\

The dimension \(f^{\alpha}\) of the Specht module \(S^{\alpha}\) is given by the following formula
\begin{equation}
\label{eq:hook}
f^{\alpha} = n!/\prod{(\textrm{hook lengths of } [\alpha])}
\end{equation}

From now on we will write \([\alpha]\) for the equivalence class of the irreducible representation \(S^{\alpha}\), \(\chi_{\alpha}\) for the irreducible character \(\chi_{S^{\alpha}}\), and \(\xi_{\alpha}\) for the character of the permutation representation \(M^{\alpha}\). Notice that the set of \(\alpha\)-tabloids form a basis for \(M^{\alpha}\), and therefore \(\xi_{\alpha}(\sigma)\), the trace of the corresponding permutation representation at \(\sigma\), is precisely the number of \(\alpha\)-tabloids fixed by \(\sigma\). 
\begin{comment}
If \(U \in [\alpha],\ V \in [\beta]\), we define \([\alpha]+[\beta]\) to be the equivalence class of \(U\oplus V\), and \([\alpha] \otimes [\beta]\) to be the equivalence class of \(U \otimes V\); since \(\chi_{U \oplus V} = \chi_{U}+\chi_{V}\) and \(\chi_{U \otimes V} = \chi_{U} \cdot \chi_{V}\), this corresponds to pointwise addition/multiplication of the corresponding characters.

For any partition \(\alpha\) of \(n\), we have \(S^{(1^{n})} \otimes S^{\alpha} \cong S^{\alpha'}\), where \(\alpha'\) is the \textrm{transpose} of \(\alpha\), the partition of \(n\) with Young diagram obtained by interchanging rows with columns in the Young diagram of \(\alpha\). Hence, \([1^{n}] \otimes [\alpha] = [\alpha']\), and \(\chi_{\alpha'} = \epsilon \cdot \chi_{\alpha}\). For example, we obtain:\\
\\
\([3,2,2]\otimes [1^{7}] = [3,2,2]' =\) \(\left[\begin{tabular}{ccc} \(\bullet\) & \(\bullet\)& \(\bullet\)\\
\(\bullet\) & \(\bullet\) &\\
\(\bullet\) & \(\bullet\) &\\
\end{tabular}\right]' = \left[\begin{tabular}{ccc} \(\bullet\) & \(\bullet\) & \(\bullet\)\\
\(\bullet\) & \(\bullet\) & \(\bullet\)\\
\(\bullet\) &&\\
\end{tabular} \right]\) \(= [3,3,1]\)\\
\\
\\
\end{comment}

We now explain how the permutation modules \(M^{\beta}\) decompose into irreducibles.

\begin{definition}
Let \(\alpha,\beta\) be partitions of \(n\). A \emph{generalized} \(\alpha\)-\emph{tableau} is produced by replacing each dot in the Young diagram of \(\alpha\) with a number between 1 and \(n\); if a generalized \(\alpha\)-tableau has \(\beta_{i}\) \(i\)'s \((1 \leq i \leq n)\) it is said to have \emph{content} \(\beta\). A generalized \(\alpha\)-tableau is said to be \emph{semistandard} if the numbers are non-decreasing along each row and strictly increasing down each column.
\end{definition}

  \begin{definition} \label{def:kostka}
  Let $\alpha, \beta$ be partitions of $n$. The {\em Kostka number}
  $K_{\alpha,\beta}$ is the number of semistandard generalized $\alpha$-tableaux
  with content $\beta$.
\end{definition}

\textit{Young's Rule} states that for any partition \(\beta\) of \(n\), the permutation module \(M^{\beta}\) decomposes into irreducibles as follows:
\[M^{\beta} \cong \oplus_{\alpha \vdash n} K_{\alpha, \beta} S^{\alpha}\]

For example, \(M^{(n-1,1)}\), which corresponds to the natural permutation action of \(S_{n}\) on \([n]\), decomposes as
\[M^{(n-1,1)} \cong S^{(n-1,1)} \oplus S^{(n)}\]
and therefore
\[\xi_{(n-1,1)} = \chi_{(n-1,1)}+1\]

Let \(V_{\alpha}\) be the subspace of \(\mathbb{C}[S_{n}]\) consisting of functions whose Fourier transform is concentrated on \([\alpha]\); equivalently, \(V_{\alpha}\) is the sum of all submodules of \(\mathbb{C}S_{n}\) isomorphic to the Specht module \(S^{\alpha}\).

We call a partition of \(n\) (or an irreducible representation of \(S_{n}\)) `fat' if its Young diagram has first row of length at least \(n-t\). Let \(\mathcal{F}_{n,t}\) denote the set of all fat partitions of \(n\); note that for \(n \geq 2t\),
\[|\mathcal{F}_{n,t}| = \sum_{s=0}^{t}p(s)\]
where \(p(s)\) denotes the number of partitions of \(s\). This grows very rapidly with \(t\), but (as will be crucial for our stability analysis) it is independent of \(n\) for \(n \geq 2t\). Note that \(\{[\alpha]:\alpha \textrm{ is fat}\}\) are precisely the irreducible constituents of the permutation module \(M^{(n-t,1^{t})}\) corresponding to the action of \(S_{n}\) on \(t\)-tuples of distinct numbers, since \(K_{\alpha,(n-t,1^{t})} \geq 1\) iff there exists a semistandard generalized \(\alpha\)-tableau of content \((n-t,1^{t})\), i.e. iff \(\alpha_{1} \geq n-t\). 

Recall from \cite{friedgutpilpel} that \(V_{t}\) is the subspace of functions whose Fourier transform is concentrated on the `fat' irreducible representations of \(S_{n}\); equivalently,
\begin{equation}
\label{eq:fatdecomposition}
V_{t} = \bigoplus_{\textrm{fat }\alpha}V_{\alpha}
\end{equation}
The projection of \(u \in \mathbb{C}[S_{n}]\) onto \(V_{\alpha}\) has \(\sigma\)-coordinate
\[P_{V_{\alpha}}(u)_{\sigma} = \frac{f^{\alpha}}{n!}\sum_{\pi \in S_{n}} u(\pi)\chi_{\alpha}(\pi \sigma^{-1})\]
and therefore the projection of \(u\) onto \(V_{t}\) has \(\sigma\)-coordinate
\begin{equation}
\label{eq:projectionformula}
P_{V_{t}}(u)_{\sigma} = \frac{1}{n!}\sum_{\textrm{fat }\alpha} f^{\alpha}\sum_{\pi \in S_{n}} u(\pi)\chi_{\alpha}(\pi \sigma^{-1})
\end{equation}

\section{Stability}
We are now in a position to prove our rough stability result:
\begin{theorem} \label{thm:stability}
Let \(t \in \mathbb{N}, c>0\) be fixed. If \(\mathcal{A} \subset S_{n}\) is a \(t\)-intersecting family with \(|\mathcal{A}| \geq c (n-t)!\), then there exists a \(t\)-coset \(\mathcal{C}\) such that \(|\mathcal{A} \setminus \mathcal{C}| \leq O((n-t-1)!)\).
\end{theorem}
In other words, if \(\mathcal{A} \subset S_{n}\) is a \(t\)-intersecting family of size at least a constant proportion of the maximum possible size \((n-t)!\), then there is some \(t\)-coset containing all but at most a \(O(1/n)\)-fraction of \(\mathcal{A}\).

To prove this, we will first prove the following weaker statement:
\begin{lemma}
\label{lemma:1coset}
Let \(t \in \mathbb{N}, c>0\) be fixed. If \(\mathcal{A} \subset S_{n}\) is a \(t\)-intersecting family of size at least \(c (n-t)!\), then there exist \(i\) and \(j\) such that all but at most \(O((n-t-1)!)\) permutations in \(\mathcal{A}\) map \(i\) to \(j\).
\end{lemma}
In other words, a large \(t\)-intersecting family is almost contained within a 1-coset. Theorem \ref{thm:stability} will follow easily from this by an inductive argument.

Given distinct \(i_{1},\ldots,i_{l}\) and distinct \(j_{1},\ldots,j_{l}\), we will write
\[\mathcal{A}_{i_{1} \mapsto j_{1}, i_{2} \mapsto j_{2}, \ldots, i_{l} \mapsto j_{l}} := \{\sigma \in \mathcal{A}: \sigma(i_{k})=j_{k}\ \forall k \in [l]\}\]

To prove Lemma \ref{lemma:1coset}, we will first observe from (\ref{eq:smallprojection}) that if \(\mathcal{A} \subset S_{n}\) is a \(t\)-intersecting family of size at least \(c(n-t)!\) then the characteristic vector \(v_{\mathcal{A}}\) of \(\mathcal{A}\) is close to the subspace \(V_{t}\) spanned by the characteristic vectors of the \(t\)-cosets. We will use this, combined with representation-theoretic arguments, to show that there exists some \(t\)-coset \(\mathcal{C}_{0}\) such that
\[|\mathcal{A} \cap \mathcal{C}_{0}| \geq \omega((n-2t)!)\]
---without loss of generality, \(\mathcal{C}_{0} = \{\sigma \in S_{n}:\ \sigma(1)=1,\ldots,\sigma(t)=t\}\), so
\[|\mathcal{A}_{1 \mapsto 1, 2 \mapsto 2, \ldots, t \mapsto t}| \geq \omega((n-2t)!)\]
Note that the average size of the intersection of \(\mathcal{A}\) with a \(t\)-coset is
\[|\mathcal{A}| / n(n-1)\ldots(n-t+1) = \Theta((n-2t)!)\]
We only know that \(\mathcal{A} \cap \mathcal{C}_{0}\) has size \(\omega\) of the average size. This statement would at first seem to weak to help us. However, for any distinct \(j_{1} \neq 1,j_{2} \neq 2,\ldots\), and \(j_{t} \neq t\), the pair of families
\[\mathcal{A}_{1 \mapsto 1, 2 \mapsto 2,\ldots,t \mapsto t},\quad \mathcal{A}_{1 \mapsto j_{1},2 \mapsto j_{2},\ldots,t \mapsto j_{t}}\]
is \(t\)-cross-intersecting, so we may compare their sizes. In detail, we will deduce from Theorem \ref{thm:tcrossintersecting} that
\[|\mathcal{A}_{1 \mapsto 1, 2 \mapsto 2,\ldots,t \mapsto t}||\mathcal{A}_{1 \mapsto j_{1},2 \mapsto j_{2},\ldots,t \mapsto j_{t}}| \leq ((n-2t)!)^{2}\]
giving \(|\mathcal{A}_{1 \mapsto j_{1},\ldots,t \mapsto j_{t}}| \leq o((n-2t)!)\). Summing over all choices of \(j_{1},\ldots,j_{t}\) will show that all but at most \(o((n-t)!)\) permutations in \(\mathcal{A}\) fix some point of \([t]\), enabling us to complete the proof.\\
\\
\textit{Proof of Lemma \ref{lemma:1coset}:}\\
Let \(\mathcal{A} \subset S_{n}\) be a \(t\)-intersecting family of size at least \(c(n-t)!\); write \(\delta = 1-c < 1\). From (\ref{eq:smallprojection}), we know that the Euclidean distance from \(v_{\mathcal{A}}\) to \(V_{t}\) is small:
\[||P_{V_{t}^{\perp}}(v_{\mathcal{A}})||^{2} \leq \delta(1+O(n^{1/6}))|\mathcal{A}|/n!\]
From (\ref{eq:projectionformula}), the projection of \(v_{\mathcal{A}}\) onto \(V_{t}\) has \(\sigma\)-coordinate:
\[P_{V_{t}}(v_{\mathcal{A}})_{\sigma} = \frac{1}{n!}\sum_{\textrm{fat }\alpha} f^{\alpha}\sum_{\pi \in \mathcal{A}} \chi_{\alpha}(\pi \sigma^{-1})\]
Write \(P_{\sigma} = P_{V_{t}}(v_{\mathcal{A}})_{\sigma}\); then
\[\frac{1}{n!}\left(\sum_{\sigma \in \mathcal{A}} (1-P_{\sigma})^{2}+\sum_{\sigma \notin \mathcal{A}} P_{\sigma}^{2}\right) \leq \delta (1+O(1/n^{1/6}))|\mathcal{A}|/n!\]
i.e.
\[\sum_{\sigma \in \mathcal{A}} (1-P_{\sigma})^{2}+\sum_{\sigma \notin \mathcal{A}} P_{\sigma}^{2} \leq \delta (1+O(1/n^{1/6}))|\mathcal{A}|\]
Choose \(C > 0: |\mathcal{A}|(1-1/n)\delta(1+C/n^{1/6}) \geq\) RHS; then the subset
\[\mathcal{S} := \{\sigma \in \mathcal{A}: (1-P_{\sigma})^{2} < \delta(1+C/n^{1/6})\}\]
has size at least \(|\mathcal{A}|/n\). Similarly, \(P_{\sigma}^{2} < 2\delta/n\) for all but at most
\[n|\mathcal{A}|(1+O(1/n))/2\]
permutations \(\sigma \notin \mathcal{A}\). Provided \(n\) is sufficiently large, \(|\mathcal{A}| \leq (n-t)!\), and therefore the subset \(\mathcal{T} = \{\sigma \notin \mathcal{A}: P_{\sigma}^{2} < 2\delta/n\}\) has size
\[|\mathcal{T}| \geq n! - (n-t)!-n(n-t)!(1+O(1/n))/2\]
The permutations \(\sigma \in \mathcal{S}\) have \(P_{\sigma}\) close to 1; the permutations \(\pi \in \mathcal{T}\) have \(P_{\pi}\) close to 0. Using only our lower bounds on the sizes of \(\mathcal{S}\) and \(\mathcal{T}\), we may prove the following:\\
\\
\textit{Claim:} There exist permutations \(\sigma \in \mathcal{S},\ \pi \in \mathcal{T}\) such that \(\sigma^{-1} \pi\) is a product of at most \(h=h(n)\) transpositions, where \(h = \sqrt{2 (t+2)(n-1) \log n}\).\\
\\
\textit{Proof of Claim:} Define the \textit{transposition graph} \(H\) to be the Cayley graph on \(S_{n}\) generated by the transpositions, i.e. \(V(H) = S_{n}\) and \(\sigma \pi \in E(H)\) iff \(\sigma^{-1} \pi\) is a transposition. We use the following isoperimetric inequality for \(H\), essentially the martingale inequality of Maurey:

\begin{theorem}
Let \(X \subset V(H)\) with \(|X| \geq \gamma n!\) where \(0 < \gamma < 1\). Then for any \(h \geq h_{0} := \sqrt{\tfrac{1}{2}(n-1)\log \tfrac{1}{\gamma}}\),
\[|N_{h}(X)| \geq \left(1-e^{-\frac{2(h-h_{0})^{2}}{n-1}}\right)n!\]
\end{theorem}
\begin{flushright}\(\square\)\end{flushright}
For a proof, see for example \cite{McD}. Applying this to the set \(\mathcal{S}\), which has \(|\mathcal{S}| \geq (1-\delta)(n-t)!/n \geq \frac{n!}{n^{t+2}}\) (provided \(n\) is sufficiently large), with \(\gamma = 1/n^{t+2},\ h = 2h_{0}\), gives \(|N_{h}(\mathcal{S})| \geq (1-n^{-(t+2)})n!\), so certainly \(N_{h}(\mathcal{S}) \cap \mathcal{T} \neq \emptyset\), proving the claim.

We now have two permutations \(\sigma \in \mathcal{A}\), \(\pi \notin \mathcal{A}\) which are `close' to one another in \(H\) (differing in only \(O(\sqrt{n \log n})\) transpositions) such that
\[P_{\sigma} > 1-\sqrt{\delta(1+C/n^{1/6})},\quad P_{\pi} < \sqrt{2\delta/n}\]
and therefore
\[P_{\sigma}-P_{\pi} > 1-\sqrt{\delta}-O(1/n^{1/12})\]
Hence, by averaging, there exist two permutations \(\rho,\tau\) that differ by just one transposition and satisfy
\[P_{\rho}-P_{\tau} > (1-\sqrt{\delta}-O(1/n^{1/12}))/h \geq \frac{1-\sqrt{\delta}-O(1/n^{1/12})}{\sqrt{2 (t+2) n \log n}}\]
i.e.
\[\sum_{\alpha \in \mathcal{F}_{n,t}} \frac{f^{\alpha}}{n!} \left(\sum_{\pi \in \mathcal{A}} \chi_{\alpha}(\pi \rho^{-1})-\sum_{\pi \in \mathcal{A}} \chi_{\alpha}(\pi \tau^{-1})\right) \geq \frac{1-\sqrt{\delta}-O(1/n^{1/12})}{\sqrt{2 (t+2) n \log n}}\]
By double translation, we may assume without loss of generality that \(\rho = \textrm{Id}\), \(\tau = (1\ 2)\). So we have:
\[\sum_{\alpha \in \mathcal{F}_{n,t}} \frac{f^{\alpha}}{n!} \left(\sum_{\pi \in \mathcal{A}} \chi_{\alpha}(\pi)-\sum_{\pi \in \mathcal{A}} \chi_{\alpha}(\pi(1\ 2))\right) \geq \frac{1-\sqrt{\delta}-O(1/n^{1/12})}{\sqrt{2 (t+2) n \log n}}\]
The above sum is over \(|\mathcal{F}_{n,t}| = \sum_{s=0}^{t}p(s)\) partitions \(\alpha\) of \(n\); this grows very rapidly with \(t\), but is independent of \(n\) for \(n \geq 2t\). By averaging, there exists some \(\alpha \in \mathcal{F}_{n,t}\) such that 
\begin{eqnarray*}
\frac{f^{\alpha}}{n!} \left(\sum_{\pi \in \mathcal{A}} \chi_{\alpha}(\pi)-\sum_{\pi \in \mathcal{A}} \chi_{\alpha}(\pi(1\ 2))\right) & \geq & \frac{1-\sqrt{\delta}-O(1/n^{1/12})}{\sqrt{2 (t+2) n \log n}\sum_{s=0}^{t}p(s)} \\
& = & \Omega(1/\sqrt{n \log n})
\end{eqnarray*}
Recall that the `fat' irreducible representations \(\{[\alpha]: \alpha \in \mathcal{F}_{n,t}\}\) are precisely the irreducible constituents of \(M^{(n-t,1^{t})}\), so very crudely, for each fat \(\alpha\),
\[f^{\alpha} \leq \dim(M^{(n-t,1^{t})}) = n(n-1)\ldots(n-t+1)\]
Hence,
\[\sum_{\pi \in \mathcal{A}} \chi_{\alpha}(\pi)-\sum_{\pi \in \mathcal{A}} \chi_{\alpha}(\pi(1\ 2)) \geq \Omega(1/\sqrt{n \log n}) (n-t)!\]
But for any \(\alpha \in \mathcal{F}_{n,t}\), we may express the irreducible character \(\chi_{\alpha}\) as a linear combination of permutation characters \(\xi_{\beta}: \beta \in \mathcal{F}_{n,t}\) using the following `determinantal formula' (see \cite{JamesKerber}). For any partition \(\alpha\) of \(n\),
\[\chi_{\alpha} = \sum_{\pi \in S_{n}}\epsilon(\pi) \xi_{\alpha - \textrm{id}+\pi}\]
Here, for \(\alpha = (\alpha_{1},\ldots,\alpha_{l}) \vdash n\), we set \(\alpha_{i}=0\ (l < i \leq n)\), we think of \(\alpha\), \(\textrm{id}\) and \(\pi\) as sequences of length \(n\), and we define addition and subtraction of these sequences pointwise. In general,
\[\alpha - \textrm{id}+\pi = (\alpha_{1}-1+\pi(1),\alpha_{2}-2+\pi(2),\ldots,\alpha_{n}-n+\pi(n))\]
will be a sequence of \(n\) integers with sum \(n\), i.e. a \textit{composition} of \(n\). If \(\lambda\) is a composition of \(n\) with all its terms non-negative, then let \(\bar{\lambda}\) be the partition of \(n\) produced by ordering the terms of \(\lambda\) in non-increasing order, and define \(\xi_{\lambda}=\xi_{\bar{\lambda}}\); if \(\lambda\) has a negative term, we define \(\xi_{\lambda}=0\). If \(\alpha \in \mathcal{F}_{n,t}\), then as \(\alpha_{1} \geq n-t\), any composition occurring in the above sum has first term at least \(n-t\), and therefore \(\xi_{\beta}\) can only occur in the above sum if \(\beta \in \mathcal{F}_{n,t}\). Observe further that since \(\alpha\) has at most \(t+1\) non-zero parts, \(\alpha_{i} = 0\) for every \(i > t+1\), and therefore any permutation \(\pi \in S_{n}\) with \(\xi_{\alpha-\textrm{id}+\pi} \neq 0\) must have \(\pi(i) \geq i\) for every  \(i > t+1\), so must fix \(t+2,t+3,\ldots,\) and \(n\). Therefore, the above sum is only over \(\pi \in S_{\{1,\ldots,t+1\}}\), i.e.
\[\chi_{\alpha} = \sum_{\pi \in S_{t+1}}\epsilon(\pi) \xi_{\alpha - \textrm{id}+\pi}\ \forall \alpha \in \mathcal{F}_{n,t}\]
Therefore, \(\chi_{\alpha}\) is a \((\pm 1)\)-linear combination of at most \((t+1)!\) permutation characters \(\xi_{\beta}\) (\(\beta \in \mathcal{F}_{n,t}\)), possibly with repeats. Hence, by averaging, there exists some \(\beta \in \mathcal{F}_{n,t}\) such that
\begin{eqnarray*}
\left|\sum_{\pi \in \mathcal{A}} \xi_{\beta}(\pi)-\sum_{\pi \in \mathcal{A}} \xi_{\beta}(\pi(1\ 2))\right| & \geq & \Omega(1/\sqrt{n \log n}) \frac{(n-t)!}{(t+1)!} \\
&= & \Omega(1/\sqrt{n \log n})(n-t)!
\end{eqnarray*}
Without loss of generality, we may assume that the above quantity is positive, i.e.
\[\sum_{\pi \in \mathcal{A}} \xi_{\beta}(\pi)-\sum_{\pi \in \mathcal{A}} \xi_{\beta}(\pi(1\ 2)) \geq \Omega(1/\sqrt{n \log n})(n-t)!\]
Let \(\mathbb{T}_{\beta}\) be the set of \(\beta\)-tabloids; the LHS is then
\begin{eqnarray*}
&& \#\{(T,\pi): T \in \mathbb{T}_{\beta},\pi \in \mathcal{A},\pi(T)=T\}\\
&-& \#\{(T,\pi): T \in \mathbb{T}_{\beta}, \pi \in \mathcal{A},\pi(1\ 2)(T)=T\}
\end{eqnarray*}
Interchanging the order of summation, this equals
\[\sum_{T \in \mathbb{T}_{\beta}} \left(\#\{\pi \in \mathcal{A}: \pi(T)=T\}-\#\{\pi \in \mathcal{A}: \pi(1\ 2)(T)=T\}\right)\]
The above summand is zero for all \(\beta\)-tabloids \(T\) with \(1\) and \(2\) in the first row of \(T\) (as then \((1\ 2)T=T\)). Write \(\beta = (n-s,\beta_{2},\ldots,\beta_{l})\), where \(0 \leq s \leq t\). The number of \(\beta\)-tabloids with 1 not in the first row is
\[s(n-1)(n-2)\ldots(n-s+1)/\prod_{i=2}^{l} \beta_{i}!\]
and therefore the number of \(\beta\)-tabloids with 1 or 2 below the first row is at most
\begin{eqnarray*}
2s(n-1)(n-2)\ldots(n-s+1)/\prod_{i=2}^{l} \beta_{i}! & \leq & 2t(n-1)(n-2)\ldots(n-s+1)\\
& = & \frac{2t(n-1)!}{(n-s)!}
\end{eqnarray*}
Hence by averaging, for one such \(\beta\)-tabloid \(T\),
\begin{eqnarray*}
\#\{\pi \in \mathcal{A}: \pi(T)=T\}-\#\{\pi \in \mathcal{A}: \pi(1\ 2)(T)=T\}\\
\geq \Omega(1/\sqrt{n \log n})\frac{(n-s)!}{2t(n-1)!}(n-t)!
\end{eqnarray*}
and therefore the number of permutations in \(\mathcal{A}\) fixing \(T\) satisfies
\[\#\{\pi \in \mathcal{A}: \pi(T)=T\}\geq \Omega(1/\sqrt{n \log n})\frac{(n-s)!}{2t(n-1)!}(n-t)!\]
Without loss of generality, we may assume that the first row of \(T\) consists of the numbers \(\{s+1,\ldots,n\}\). There are \(\beta_{2}!\beta_{3}!\ldots \beta_{l}! \leq s! \leq t!\) permutations of \([s]\) fixing the \(2^{\textrm{nd}}\),\(3^{\textrm{rd}}\),\(\ldots\), and \(l^{\textrm{th}}\) rows of \(T\); any permutation fixing \(T\) must agree with one of these permutations on \([s]\). Hence, there exists a permutation \(\rho\) of \([s]\) such that at least 
\[\Omega(1/\sqrt{n \log n})\frac{(n-s)!(n-t)!}{2t(n-1)!t!}\]
permutations in \(\mathcal{A}\) agree with \(\rho\) on \([s]\). Without loss of generality, we may assume that \(\rho = \textrm{Id}_{[s]}\), so the number of permutations in \(\mathcal{A}\) fixing \([s]\) pointwise satisfies
\begin{eqnarray*}
|\mathcal{A}_{1 \mapsto 1,\ldots,s \mapsto s}| & \geq &\Omega(1/\sqrt{n \log n})\frac{(n-s)!(n-t)!}{2t(n-1)!t!}\\
& = & \Omega(1/\sqrt{n \log n})\frac{(n-s)!(n-t)!}{(n-1)!}
\end{eqnarray*}
We may write \(\mathcal{A}_{1 \mapsto 1,\ldots,s \mapsto s}\) as a disjoint union
\[\mathcal{A}_{1 \mapsto 1,\ldots,s \mapsto s} = \bigcup_{j_{s+1},\ldots,j_{t} > s \textrm{ distinct }} \mathcal{A}_{1 \mapsto 1 \ldots, s \mapsto s, s+1 \mapsto j_{s+1},\ldots,t \mapsto j_{t}}\]
and there are \((n-s)(n-s-1)\ldots(n-t+1)\) choices of \(j_{s+1},\ldots,j_{t}\), so by averaging, there exists a choice such that
\[|\mathcal{A}_{1 \mapsto 1 \ldots, s \mapsto s, s+1 \mapsto j_{s+1},\ldots,t \mapsto j_{t}}| \geq \Omega(1/\sqrt{n \log n})\frac{((n-t)!)^{2}}{(n-1)!}\]
By translation, we may assume without loss of generality that \(j_{k}=k\) for each \(k\), so
\begin{eqnarray*}
|\mathcal{A}_{1 \mapsto 1, 2 \mapsto 2, \ldots, t \mapsto t}| & \geq & \Omega(1/\sqrt{n \log n})\frac{((n-t)!)^{2}}{(n-1)!}\\
& = & \Omega(\sqrt{n/\log n}) (n-2t)!\\
& = & \omega((n-2t)!)
\end{eqnarray*}
We will use this to show that the number of permutations in \(\mathcal{A}\) with no fixed point in \([t]\) is small. We may write
\[\mathcal{A} \setminus (\mathcal{A}_{1\mapsto 1} \cup \ldots \cup \mathcal{A}_{t \mapsto t}) = \bigcup_{j_{1},\ldots,j_{t} \textrm{ distinct }: j_{k} \neq k \ \forall k \in [t]} \mathcal{A}_{1 \mapsto j_{1}, \ldots, t \mapsto j_{t}}\]
We now show that each \(\mathcal{A}_{1 \mapsto j_{1}, \ldots, t \mapsto j_{t}}\) is small using Theorem \ref{thm:tcrossintersecting}. Let \(J = \{j_{1},\ldots,j_{t}\}\). Notice that \(\mathcal{E}:=\mathcal{A}_{1 \mapsto 1,\ldots, t\mapsto t},\ \mathcal{F}:=\mathcal{A}_{1 \mapsto j_{1}, \ldots, t \mapsto j_{t}}\) is a \(t\)-cross-intersecting pair of families, so for any \(\sigma \in \mathcal{E}\) and \(\pi \in \mathcal{F}\), there are \(t\) distinct points \(i_{1},i_{2}\ldots,i_{t} > t\) such that \(\sigma(i_{k})=\pi(i_{k}) \notin [t] \cup J\) for each \(k \in [t]\). But then
\[(1\ j_{1})(2\ j_{2}) \ldots (t\ j_{t})\pi(i_{k}) = \sigma(i_{k})\quad \textrm{for each } k \in [t]\]
so letting \(\mathcal{G}:=(1\ j_{1})(2\ j_{2})\ldots(t\ j_{t})\mathcal{F}\), the pair of families \(\mathcal{E},\mathcal{G}\) fix \([t]\) pointwise and \(t\)-cross-intersect on \(\{t+1,t+2,\ldots,n\}\). Deleting \(1,\ldots,t\) we obtain a \(t\)-cross-intersecting pair \(\mathcal{E}',\mathcal{G}'\) of subsets of \(S_{\{t+1,\ldots,n\}}\). By Theorem \ref{thm:tcrossintersecting},
\[|\mathcal{A}_{1 \mapsto 1,\ldots,t \mapsto t}||\mathcal{A}_{1 \mapsto j_{1},\ldots,t \mapsto j_{t}}| = |\mathcal{E}||\mathcal{G}| = |\mathcal{E}'||\mathcal{G}'| \leq ((n-2t)!)^{2}\]
Since 
\[|\mathcal{A}_{1 \mapsto 1,\ldots,t \mapsto t}| \geq \omega((n-2t)!)\]
we have
\[|\mathcal{A}_{1 \mapsto j_{1}, \ldots, t \mapsto j_{t}}| \leq o((n-2t)!)\]
There are \(\leq n(n-1)(n-2)\ldots(n-t+1)\) possible choices of \(j_{1},\ldots,j_{t}\), and therefore the number of permutations in \(\mathcal{A}\) with no fixed point in \([t]\) satisfies
\begin{eqnarray*}
|\mathcal{A}\setminus (\mathcal{A}_{1\mapsto 1} \cup \mathcal{A}_{2 \mapsto 2} \cup \ldots \cup \mathcal{A}_{t \mapsto t})| & \leq & o((n-2t)!) n(n-1)\ldots(n-t+1)\\
& = & o((n-t)!)
\end{eqnarray*}
Since \(|\mathcal{A}| \geq c(n-t)!\), we have
\[|\mathcal{A}_{1\mapsto 1} \cup \mathcal{A}_{2 \mapsto 2} \cup \ldots \cup \mathcal{A}_{t\mapsto t}| \geq (c-o(1))(n-t)!\]
By averaging, there exists some \(i \in [t]\) such that
\[|\mathcal{A}_{i \mapsto i}| \geq (c-o(1))(n-t)!/t\]
We may assume that \(i=1\), so \(|\mathcal{A}_{1 \mapsto 1}| \geq (c-o(1))(n-t)!/t\). Now, using the same trick as before, we may use Theorem \ref{thm:tcrossintersecting} to show that \(|\mathcal{A} \setminus \mathcal{A}_{1 \mapsto 1}| \leq O((n-t-1)!)\). Indeed, write \(\mathcal{A} \setminus \mathcal{A}_{1 \mapsto 1}\) as a disjoint union
\[\mathcal{A} \setminus \mathcal{A}_{1 \mapsto 1} = \bigcup_{j \neq 1} \mathcal{A}_{1 \mapsto j}\]
We will show that each \(\mathcal{A}_{1 \mapsto j}\) is small. Notice as before that the pair of families \(\mathcal{A}_{1\mapsto 1},\ (1\ j)\mathcal{A}_{1 \mapsto j}\) fixes 1 and \(t\)-cross-intersects on the domain \(\{2, \ldots ,n\}\), so Theorem \ref{thm:tcrossintersecting} gives
\[|\mathcal{A}_{1 \mapsto 1}||\mathcal{A}_{1 \mapsto j}| \leq ((n-t-1)!)^{2}\]
Since \(|\mathcal{A}_{1 \mapsto 1}| \geq \Omega((n-t)!)\), we obtain \(|\mathcal{A}_{1 \mapsto j}| \leq O((n-t-2)!)\), and therefore
\[|\mathcal{A} \setminus \mathcal{A}_{1 \mapsto 1}| = \sum_{j \neq 1}|\mathcal{A}_{1 \mapsto j}| \leq O((n-t-1)!)\]
proving Lemma \ref{lemma:1coset}.
\begin{flushright}\(\square\)\end{flushright}
\textit{Proof of Theorem \ref{thm:stability}:}\\
By induction on \(t\). The \(t=1\) case is the same as that of Lemma \ref{lemma:1coset}. Assume the theorem is true for \(t-1\); we will prove it for \(t\). Let \(\mathcal{A} \subset S_{n}\) be a \(t\)-intersecting family of size at least \(c(n-t)!\). By Lemma \ref{lemma:1coset}, there exist \(i\) and \(j\) such that \(|\mathcal{A} \setminus \mathcal{A}_{i \mapsto j}| \leq  O((n-t-1)!)\). Without loss of generality we may assume that \(i=j=1\), so \(|\mathcal{A} \setminus \mathcal{A}_{1 \mapsto 1}| \leq  O((n-t-1)!)\). Hence, \(|\mathcal{A}_{1\mapsto 1}| \geq |\mathcal{A}| - O((n-t-1)!)\). Deleting \(1\) from each permutation in \(\mathcal{A}_{1\mapsto 1}\), we obtain a \((t-1)\)-intersecting family \(\mathcal{A}' \subset S_{\{2,3,\ldots,n\}}\) of size \(\geq (c-O(1/n))(n-t)!\). Choose any positive constant \(c' < c\); then provided \(n\) is sufficiently large, we have \(|\mathcal{A}'| \geq c'(n-t)!\). By the induction hypothesis, there exists a \((t-1)\)-coset \(\mathcal{C}'\) of \(S_{2,\ldots,n}\) such that \(|\mathcal{A}'\setminus \mathcal{C}'| \leq O((n-t-1)!)\). Then if \(\mathcal{C}\) is the \(t\)-coset obtained from \(\mathcal{C}'\) by adjoining \(1 \mapsto 1\), we have \(|\mathcal{A} \setminus \mathcal{C}| \leq O((n-t-1)!)\). This completes the induction and proves Theorem \ref{thm:stability}.
\begin{flushright}\(\square\)\end{flushright}

We now use our rough stability result to prove an exact stability result. First, we need some more definitions.

Let \(d_{n}\) be the number of \textit{derangements} of \([n]\) (permutations of \([n]\) without fixed points). It is well known that \(d_{n} = (1/e+o(1))n!\).

Following Cameron and Ku \cite{cameron}, given a permutation \(\rho \in S_{n}\) and \(i \in [n]\), we define the \textit{i-fix} of \(\rho\) to be the permutation \(\rho_{i}\) which fixes \(i\), maps the preimage of \(i\) to the image of \(i\), and agrees with \(\rho\) at all other points of \([n]\), i.e.
\[\rho_{i}(i) = i;\ \rho_{i}(\rho^{-1}(i)) = \rho(i);\ \rho_{i}(k)=\rho(k)\ \forall k \neq i,\rho^{-1}(i)\]
In other words, \(\rho_{i} = \rho (\rho^{-1}(i)\ i)\). We inductively define
\[\rho_{i_{1},\ldots,i_{l}} = (\rho_{i_{1},\ldots,i_{l-1}})_{i_{l}}\]
Notice that if \(\sigma\) fixes \(j\), then \(\sigma\) agrees with \(\rho_{j}\) wherever it agrees with \(\rho\).

\begin{theorem}\label{thm:hiltonmilnertype}
For \(n\) sufficiently large depending on \(t\), if \(\mathcal{A} \subset S_{n}\) is a \(t\)-intersecting family which is not contained within a \(t\)-coset, then \(\mathcal{A}\) is no larger than the family
\begin{eqnarray*}
\mathcal{D}& = &\{\sigma \in S_{n}:\ \sigma(i)=i\ \forall i \leq t,\ \sigma(j)=j\ \textrm{for some}\ j > t+1\} \\
&&\cup \{(1\ t+1),(2\ t+1),\ldots,(t \ t+1)\}
\end{eqnarray*}
which has size \((n-t)!-d_{n-t}-d_{n-t-1}+t = (1-1/e+o(1))(n-t)!\). If \(\mathcal{A}\) is the same size as \(\mathcal{D}\), then \(\mathcal{A}\) is a double translate of \(\mathcal{D}\), i.e. \(\mathcal{A}=\pi \mathcal{D} \tau\) for some \(\pi,\tau \in S_{n}\).
\end{theorem}

\begin{proof}
Suppose \(\mathcal{A} \subset S_{n}\) is a \(t\)-intersecting family which is not contained within a \(t\)-coset, and has size
\[|\mathcal{A}| \geq (n-t)!-d_{n-t}-d_{n-t-1}+t = (1-1/e+o(1))(n-t)!.\]
Applying Theorem \ref{thm:stability} with any constant \(c\) such that \(0 < c < 1-1/e\), we see that (provided \(n\) is sufficiently large) there exists a \(t\)-coset \(\mathcal{C}\) such that 
\[|\mathcal{A} \setminus \mathcal{C}| \leq O(1/n)(n-t)!\]
By double translation, without loss of generality we may assume that \(\mathcal{C} = \{\sigma \in S_{n}: \sigma(1)=1,\ldots,\sigma(t)=t\}\). We have:
\begin{eqnarray}
\label{eq:boundC}
|\mathcal{A} \cap \mathcal{C}|& \geq& (n-t)!-d_{n-t}-d_{n-t-1}+t-O(1/n)(n-t)!\nonumber \\
& = &(1-1/e+o(1))(n-t)!
\end{eqnarray}

We now claim that every permutation in \(\mathcal{A}\setminus \mathcal{C}\) fixes exactly \(t-1\) points of \([t]\). Suppose for a contradiction that \(\mathcal{A}\) contains a permutation \(\tau\) fixing at most \(t-2\) points of \([t]\). Then every permutation in \(\mathcal{A} \cap \mathcal{C}\) must agree with \(\tau\) on at least 2 points of \(\{t+1,\ldots,n\}\), so
\[|\mathcal{A} \cap \mathcal{C}| \leq {n-t \choose 2}(n-t-2)! = \tfrac{1}{2}(n-t)!\]
contradicting (\ref{eq:boundC}), provided \(n\) is sufficiently large.

Since we are assuming that \(\mathcal{A}\) is not contained within a \(t\)-coset, \(\mathcal{A} \setminus \mathcal{C}\) contains some permutation \(\tau\); \(\tau\) must fix all points of \([t]\) except for one. By double translation, we may assume that \(\tau = (1\ t+1)\). We will show that under these hypotheses, \(\mathcal{A}=\mathcal{D}\).

Every permutation in \(\mathcal{A} \cap \mathcal{C}\) must \(t\)-intersect \((1 \ t+1)\) and must therefore have at least one fixed point \(> t+1\), i.e. \(\mathcal{A} \cap \mathcal{C}\) is a subset of the family
\[\mathcal{E}:=\{\sigma \in S_{n}: \sigma(i)=i \ \forall i \in [t], \ \sigma(j)=j \textrm{ for some } j > t+1\}\]
which has size
\[(n-t)!-d_{n-t}-d_{n-t-1}\]
We now make the following observation:\\
\\
\textit{Claim:} \(\mathcal{A} \setminus \mathcal{C}\) may only contain the transpositions \(\{(i\ t+1)\ : i \in [t]\}\).\\
\\
\textit{Proof of Claim:}\\
Suppose for a contradiction that \(\mathcal{A} \setminus \mathcal{C}\) contains a permutation \(\rho\) not of this form. Then \(\rho(j) \neq j\) for some \(j \geq t+2\).  We will show that there are at least \(d_{n-t-1}\) permutations in \(\mathcal{E}\) which fix \(j\) and disagree with \(\rho\) at every point of \(\{t+1,t+2,\ldots,n\}\), and therefore cannot \(t\)-intersect \(\rho\). Let \(l\) be the unique point of \([t]\) not fixed by \(\rho\). If \(\sigma\) fixes both \(l\) and \(j\), then \(\sigma\) agrees with \(\rho_{j,l} = (\rho_{j})_{l}\) wherever it agrees with \(\rho\). Notice that \(\rho_{j,l}\) fixes \(1,2,\ldots,t\) and \(j\). There are exactly \(d_{n-t-1}\) permutations in \(\mathcal{E}\) which fix \(j\) and disagree with \(\rho_{j,l}\) at every point of \(\{t+1,t+2,\ldots,n\}\setminus\{j\}\); each disagrees with \(\rho\) at every point of \(\{t+1,t+2,\ldots,n\}\). So none \(t\)-intersect \(\rho\), so none are in \(\mathcal{A}\), and therefore
\[|\mathcal{A} \cap \mathcal{C}| \leq |\mathcal{E}| - d_{n-t-1} = (n-t)!-d_{n-t}-2d_{n-t-1}\]
Since we are assuming that \(|\mathcal{A}| \geq (n-t)!-d_{n-t}-d_{n-t-1}+t\), this means that
\[|\mathcal{A} \setminus \mathcal{C}| \geq d_{n-t-1}+t = (1/e+o(1))(n-t-1)!\]
Notice that for any \(m \leq n\) we have the following trivial upper bound on the size of an \(m\)-intersecting family \(\mathcal{H} \subset S_{n}\):
\[|\mathcal{H}| \leq {n \choose m}(n-m)! = n!/m!\]
since every permutation in \(\mathcal{H}\) must agree with a fixed permutation in \(\mathcal{H}\) in at least \(m\) places.

Hence, \(\mathcal{A}\setminus \mathcal{C}\) cannot be \((\log n)\)-intersecting and therefore contains two permutations \(\pi,\tau\) agreeing on at most \(\log n\) points. The number of permutations fixing \([t]\) pointwise and agreeing with both \(\pi\) and \(\tau\) at one of these \(\log n\) points is therefore at most \((\log n)(n-t-1)!\). All other permutations in \(\mathcal{A} \cap \mathcal{C}\) agree with \(\pi\) and \(\tau\) at two separate points of \(\{t+1,\ldots,n\}\), and by the above argument, the same holds for \(\pi_{p}\) and \(\tau_{q}\), where \(p\) and \(q\) are the points of \([t]\) shifted by \(\pi\) and \(\tau\) respectively. The number of permutations in \(\mathcal{C}\) that agree with \(\pi_{p}\) and \(\tau_{q}\) at two separate points of \(\{t+1,\ldots,n\}\) is at most \(((1-1/e)^{2} + o(1))(n-t)!\) (it is easily checked that given two fixed permutations, the probability that a uniform random permutation agrees with them at separate points is at most \((1-1/e)^{2} + o(1)\)), which implies that
\begin{eqnarray*}
|\mathcal{A} \cap \mathcal{C}| & \leq & ((1-1/e)^{2} + o(1))(n-t)! + (\log n) (n-t-1)!\\
& = & ((1-1/e)^{2}+o(1))(n-t)!
\end{eqnarray*}
contradicting (\ref{eq:boundC}), provided \(n\) is sufficiently large. This proves the claim.

Since we are assuming \(|\mathcal{A}| \geq |\mathcal{E}| + t\), we must have equality, so \(\mathcal{A}=\mathcal{D}\), proving Theorem \ref{thm:hiltonmilnertype}.
\end{proof}

Similar arguments give the following stability results for \(t\)-cross-intersecting families. Say two pairs of families \((\mathcal{A},\mathcal{B})\), \((\mathcal{C},\mathcal{D})\) in \(S_{n}\) are \textit{isomorphic} if there exist permutations \(\pi,\rho \in S_{n}\) such that \(\mathcal{A} = \pi \mathcal{C} \rho\) and \(\mathcal{B} = \pi \mathcal{D} \rho\). We have:

\begin{theorem} \label{thm:stabilitycrossmin}
F \(n\) sufficiently large depending on \(t\), if \(\mathcal{A},\mathcal{B} \subset S_{n}\) are \(t\)-cross-intersecting but not both contained within the same \(t\)-coset, then
\[\min(|\mathcal{A}|,|\mathcal{B}|) \leq (n-t)!-d_{n-t}-d_{n-t-1}+t\]
with equality iff \((\mathcal{A},\mathcal{B})\) is isomorphic to the pair of families
\[\{\sigma: \sigma(i)=i\ \forall i \leq t,\ \sigma(j)=\tau(j)\ \textrm{for some}\ j > t+1\} \cup \{(i\ t+1):\ i \in [t]\}\]
\[\{\sigma: \sigma(i)=i\ \forall i \leq t,\ \sigma(j)=j\ \textrm{for some}\ j > t+1\} \cup \{(1i)\tau(1i): i \in [t]\}\]
where \(\tau(1) \neq 1\) and if \(t \geq 2\), \(\tau\) fixes \(2,3,\ldots,t\) and at least two points \(> t+1\), whereas if \(t=1\), \(\tau\) intersects \((1\ 2)\). 
\end{theorem}

\begin{theorem}
For \(n\) sufficiently large depending on \(t\), if \(\mathcal{A},\mathcal{B} \subset S_{n}\) are \(t\)-cross-intersecting but not both contained within the same \(t\)-coset, then
\[|\mathcal{A}||\mathcal{B}| \leq ((n-t)!-d_{n-t}-d_{n-t-1})((n-t)!+t)\]
with equality iff \((\mathcal{A},\mathcal{B})\) is isomorphic to the pair of families
\[\{\sigma \in S_{n}: \sigma(i)=i\ \forall i \leq t,\ \sigma(j)=j\ \textrm{for some}\ j > t+1\}\]
\[\{\sigma \in S_{n}: \sigma(i)=i\ \forall i \leq t\} \cup \{(1 \ t+1),(2\ t+1),\ldots,(t \ t+1)\}\]
\end{theorem}

The proofs are very similar to the proof of Theorem \ref{thm:hiltonmilnertype}, and we omit them.

\section{The Alternating Group} \label{section:alternating}
We now turn our attention to the alternating group \(A_{n}\), the index-2 subgroup of \(S_{n}\) consisting of the even permutations of \(\{1,2,\ldots,n\}\). The following may be deduced from the proof of the Deza-Frankl conjecture in \cite{friedgutpilpel}:

\begin{theorem} \label{thm:alternatingdezafrankl}
For \(n\) sufficiently large depending on \(t\), if \(\mathcal{A} \subset A_{n}\) is \(t\)-intersecting, then \(|\mathcal{A}| \leq (n-t)!/2\).
\end{theorem}
\noindent \textit{Remark:} This implies the Deza-Frankl conjecture. To see this, let \(\mathcal{A} \subset S_{n}\) be \(t\)-intersecting; then \(\mathcal{A} \cap A_{n}\) and \((\mathcal{A} \setminus A_{n})(1\ 2)\) are both \(t\)-intersecting families of permutations in \(A_{n}\), so by Theorem \ref{thm:alternatingdezafrankl}, both have size at most \((n-t)!/2\). Hence,
\[|\mathcal{A}| = |\mathcal{A} \cap A_{n}| + |\mathcal{A} \setminus A_{n}| \leq (n-t)!\]
\begin{proof}
Recall that in \cite{friedgutpilpel}, we constructed a weighted graph \(Y_{\textrm{even}}\) which was a real linear combination of Cayley graphs on \(S_{n}\) generated by conjugacy-classes of \textit{even} permutations with less than \(t\) fixed points, and whose matrix of weights had maximum eigenvalue \(1\) and minimum eigenvalue
\[\omega_{n,t} = -\frac{1}{n(n-1)\ldots(n-t+1)-1}\]
Clearly, \(Y_{\textrm{even}}\) has no (non-zero) edges between \(A_{n}\) and \(S_{n} \setminus A_{n}\). Let \(Y_{1}\) be the weighted subgraph of \(Y_{\textrm{even}}\) induced on \(A_{n}\), and \(Y_{2}\) the weighted subgraph induced on \(S_{n} \setminus A_{n}\). Notice that the map
\begin{eqnarray*}
\phi: A_{n} & \to & S_{n} \setminus A_{n};\\
\sigma & \mapsto & (1\ 2)\sigma
\end{eqnarray*}
is a graph isomorphism from \(Y_{1}\) to \(Y_{2}\). To see this, note that
\[\phi(\sigma)(\phi(\pi))^{-1}= ((1\ 2) \sigma) ((1 \ 2)\pi)^{-1} = (1\ 2)\sigma \pi^{-1}(1 \ 2)\]
which is conjugate to \(\sigma \pi^{-1}\). Since \(Y_{\textrm{even}}\) is a linear combination of Cayley graphs generated by conjugacy-classes of \(S_{n}\), the edge \(\phi(\sigma) \phi(\pi)\) has the same weight in \(Y_{\textrm{even}}\) as the edge \(\sigma \pi\). Hence, \(Y_{\textrm{even}}\) is a disjoint union of the two isomorphic subgraphs \(Y_{1}\) and \(Y_{2}\), so the eigenvalues of \(Y_{\textrm{even}}\) are the same as those of \(Y_{1}\) (with double the multiplicities). Applying Theorem \ref{thm:weightedhoffman} to \(Y_{1}\) proves Theorem \ref{thm:alternatingdezafrankl}.
\end{proof}

Our next aim is to show that equality holds in Theorem \ref{thm:alternatingdezafrankl} only if \(\mathcal{A}\) is a coset of the stabilizer of \(t\) points. As for \(S_{n}\), we will call these families the `\(t\)-cosets of \(A_{n}\)'.

Let \(W_{t}\) be the subspace of \(\mathbb{C}[A_{n}]\) spanned by the characteristic vectors of the \(t\)-cosets of \(A_{n}\). It is easily checked that \(W_{t}\) is the direct sum of the 1 and \(\omega_{n,t}\)-eigenspaces of \(Y_{1}\). Hence, by Theorem \ref{thm:weightedhoffman}, if equality holds in Theorem \ref{thm:alternatingdezafrankl}, then the characteristic vector \(v_{\mathcal{A}}\) of \(\mathcal{A}\) lies in the subspace \(W_{t}\).\\

We would like to show that the Boolean functions which are linear combinations of the characteristic functions of the \(t\)-cosets of \(A_{n}\) are precisely the characteristic functions of the disjoint unions of \(t\)-cosets of \(A_{n}\). To do this for \(S_{n}\) in \cite{friedgutpilpel}, it was first proved that if a non-negative function \(f:S_{n} \to \mathbb{R}_{\geq 0}\) is a linear combination of the characteristic functions of the \(t\)-cosets of \(S_{n}\), then it can be expressed as a linear combination of them with non-negative coefficients. However, this is not true in the case of \(A_{n}\), even for \(t=1\):\\
\\
\textit{Claim:} There exists a non-negative function in \(W_{1}\) which cannot be written as a non-negative linear combination of the characteristic functions of the 1-cosets of \(A_{n}\).\\
\\
\textit{Proof of Claim:} Let \(w_{i \mapsto j}\) be the characteristic function of the \(1\)-coset \(\{\sigma \in A_{n}:\ \sigma(i)=j\}\). We say a real \(n \times n\) matrix \(B\) \textit{represents} a function \(f \in W_{1}\) if \(f\) can be written as a linear combination of \(w_{i \mapsto j}\)'s with coefficients given by the matrix \(B\), i.e.
\[f= \sum_{i,j=1}^{n} b_{i,j}w_{i \mapsto j}\]
or equivalently,
\[f(\sigma) = \sum_{i=1}^{n} b_{i,\sigma(i)}\quad \forall \sigma \in A_{n}\]
It is easy to see that, provided \(n \geq 4\), any function \(f \in W_{1}\) has a unique extension to a function \(\tilde{f} \in V_{1}\). Hence, if \(B\) and \(C\) are two matrices both representing \(f\), they must both represent the same function \(\tilde{f}:S_{n} \to \mathbb{R}\), and therefore
\[\sum_{i=1}^{n} b_{i,\sigma(i)} = \sum_{i=1}^{n}c_{i, \sigma(i)} \quad \forall \sigma \in S_{n}\]
Now let \(f\) be the function represented by the matrix
\begin{displaymath}
B = \left(\begin{array}{cccccc}1 & -1/2 & 1 & 1 & \ldots & 1\\
-1/2 & 1 & 1 & 1 & \ldots & 1\\
1 & 1 & 0 & 1 & \ldots & 1\\
\vdots & \vdots & & \ddots && \vdots\\
1 & 1 & \ldots &&& 0\end{array}\right)\end{displaymath}
This takes only non-negative values on \(A_{n}\), since
\[\sum_{i=1}^{n} b_{i,\sigma(i)} \geq 0\quad \forall \sigma \in A_{n}\]
but if \(\tau\) is the transposition \((1\ 2)\), then
\[\sum_{i=1}^{n}b_{i,\tau(i)} = -1\]
Hence, any matrix \(C\) representing the same function as \(B\) must also have
\[\sum_{i=1}^{n}c_{i,\tau(i)} = -1\]
and therefore cannot have non-negative entries. Therefore, \(f\) is a non-negative function in \(W_{1}\) that cannot be written as a non-negative linear combination of the \(w_{i \mapsto j}\)'s, proving the claim. 

Instead, we obtain our desired characterization of equality in Theorem \ref{thm:alternatingdezafrankl} from a stability result for \(t\)-intersecting families in \(A_{n}\).

Let \(e_{n},o_{n}\) denote the number of respectively even/odd derangements of \([n]\). It is well known that \(e_{n}-o_{n} = (-1)^{n-1}(n-1)\ \forall n \in \mathbb{N}\); combining this with the fact that \(d_{n} = (1/e+o(1))n!\) gives \(e_{n} = (1/(2e) + o(1))n!,\ o_{n} = (1/(2e) + o(1))n!\).

We now prove the following analogue of Theorem \ref{thm:hiltonmilnertype}:

\begin{theorem}\label{thm:hiltonmilnertypeAn}
For \(n\) sufficiently large depending on \(t\), if \(\mathcal{A} \subset A_{n}\) is a \(t\)-intersecting family which is not contained within a \(t\)-coset of \(A_{n}\), then \(\mathcal{A}\) cannot be larger than the family
\begin{eqnarray*}
\mathcal{B}& = &\{\sigma \in A_{n}: \sigma(i)=i\ \forall i \leq t,\ \sigma(j)=(n-1\ n)(j)\ \textrm{for some}\ j > t+1\} \\
&&\cup \{(1\ t+1)(n-1\ n),(2\ t+1)(n-1\ n),\ldots,(t \ t+1)(n-1\ n)\}
\end{eqnarray*}
which has size \((n-t)!/2-o_{n-t}-o_{n-t-1}+t = (1-1/e+o(1))(n-t)!/2\). If \(\mathcal{A}\) is the same size as \(\mathcal{B}\), then \(\mathcal{A}\) is a double translate of \(\mathcal{B}\), meaning that \(\mathcal{A} = \pi \mathcal{B} \tau\) for some \(\pi,\tau \in A_{n}\).
\end{theorem}
\begin{proof}
Let \(\mathcal{A} \subset A_{n}\) be a \(t\)-intersecting family which is not contained within a \(t\)-coset of \(A_{n}\) and has size
\[|\mathcal{A}| \geq (n-t)!/2-o_{n-t}-o_{n-t-1}+t = (1-1/e+o(1))(n-t)!/2.\]
Applying Theorem \ref{thm:stability} with any constant \(c\) such that \(0 < c < (1-1/e)/2\), we see that (provided \(n\) is sufficiently large) there exists a \(t\)-coset \(\mathcal{C}\) such that
\[|\mathcal{A} \setminus \mathcal{C}| \leq O(1/n)(n-t)!\]
By double translation, without loss of generality we may assume that \(\mathcal{C} = \{\sigma \in A_{n}: \sigma(1)=1,\ldots,\sigma(t)=t\}\). We have:
\begin{eqnarray}
\label{eq:alternatingboundC}
|\mathcal{A} \cap \mathcal{C}| & \geq &(n-t)!/2-o_{n-t}-o_{n-t-1}+t-O(1/n)(n-t)! \nonumber \\
&=& (1-1/e+o(1))(n-t)!/2
\end{eqnarray}

We now claim that every permutation in \(\mathcal{A}\setminus \mathcal{C}\) fixes exactly \(t-1\) points of \([t]\). Suppose for a contradiction that \(\mathcal{A}\) contains a permutation \(\tau\) fixing at most \(t-2\) points of \([t]\). Then every permutation in \(\mathcal{A} \cap \mathcal{C}\) must agree with \(\tau\) on at least 2 points of \(\{t+1,\ldots,n\}\), so
\[|\mathcal{A} \cap \mathcal{C}| \leq {n-t \choose 2}(n-t-2)!/2 = \tfrac{1}{2}(n-t)!/2\]
contradicting (\ref{eq:alternatingboundC}), provided \(n\) is sufficiently large.

Since we are assuming that \(\mathcal{A}\) is not contained within a \(t\)-coset, \(\mathcal{A} \setminus \mathcal{C}\) contains some permutation \(\tau\); \(\tau\) must fix all points of \([t]\) except for one. By double translation, we may assume that \(\tau = (1\ t+1)(n-1\ n)\). We will show that under these hypotheses, \(\mathcal{A}=\mathcal{B}\). Every permutation in \(\mathcal{A} \cap \mathcal{C}\) must agree with \((n-1\ n)\) at some point \(\geq t+2\), i.e. \(\mathcal{A} \cap \mathcal{C}\) is a subset of the family
\[\mathcal{E}:=\{\sigma \in A_{n}: \sigma(i)=i \ \forall i \in [t], \ \sigma(j)=(n-1\ n)(j) \textrm{ for some } j \geq t+2\}\]
which has size
\[(n-t)!/2-o_{n-t}-o_{n-t-1}\]
We now make the following observation:\\
\\
\textit{Claim:} \(\mathcal{A} \setminus \mathcal{C}\) may only contain the permutations \(\{(i\ t+1)(n-1\ n):\ i \in [t]\}\).\\
\\
\textit{Proof of Claim:}\\
Suppose for a contradiction that \(\mathcal{A} \setminus \mathcal{C}\) contains a permutation \(\rho\) not of this form. Then \(\rho(j) \neq (n-1\ n)(j)\) for some \(j \geq t+2\), so by a very similar argument to in the proof of Theorem \ref{thm:hiltonmilnertype}, there are at least \(\min(e_{n-t-1},o_{n-t-1})\) even permutations which fix \(1,2,\ldots,t\) and agree with \((n-1 \ n)\) at \(j\) (and are therefore in \(\mathcal{E}\)) and also disagree with \(\rho\) at all points of \(\{t+1,t+2,\ldots,n\}\setminus\{j\}\). Since \(\rho\) has exactly \(t-1\) fixed points in \([t]\), none of these permutations can \(t\)-intersect \(\rho\), and therefore
\begin{eqnarray*}
|\mathcal{A} \cap \mathcal{C}| &\leq& |\mathcal{E}| - \min(e_{n-t-1},o_{n-t-1})\\
& =& (n-t)!-o_{n-t}-o_{n-t-1}-\min(e_{n-t-1},o_{n-t-1})
\end{eqnarray*}
Since we are assuming that \(|\mathcal{A}| \geq (n-t)!-o_{n-t}-o_{n-t-1}+t\), this means that
\[|\mathcal{A} \setminus \mathcal{C}| \geq \min(e_{n-t-1},o_{n-t-1})+t = (1/e+o(1))(n-t-1)!/2\]
Notice that for any \(m < n\) we have the following trivial upper bound on the size of an \(m\)-intersecting family \(\mathcal{H} \subset A_{n}\):
\[|\mathcal{H}| \leq {n \choose m}(n-m)!/2 = n!/(2m!)\]
since every permutation in \(\mathcal{H}\) must agree with a fixed permutation in \(\mathcal{H}\) in at least \(m\) places.

Hence, \(\mathcal{A}\setminus \mathcal{C}\) cannot be \((\log n)\)-intersecting and therefore contains two permutations \(\pi,\tau\) agreeing on at most \(\log n\) points. The number of permutations in \(\mathcal{C}\) which agree with \(\pi\) and \(\tau\) at one of these \(\log n\) points is clearly at most \((\log n)(n-t-1)!/2\). All other permutations in \(\mathcal{A} \cap \mathcal{C}\) agree with \(\pi\) and \(\tau\) at two separate points of \(\{t+1,\ldots,n\}\), and therefore the same holds for \(\pi_{p}\) and \(\tau_{q}\), where \(p\) and \(q\) are the unique points of \([t]\) shifted by \(\pi\) and \(\tau\) respectively. The number of permutations in \(\mathcal{C}\) that agree with \(\pi_{p}\) and \(\tau_{q}\) at two separate points of \(\{t+1,\ldots,n\}\) is at most \(((1-1/e)^{2} + o(1))(n-t)!/2\) (it is easily checked that given two fixed permutations, the probability that a uniform random even permutation agrees with them at separate points is at most \((1-1/e)^{2} + o(1)\)), which implies that
\begin{eqnarray*}
|\mathcal{A} \cap \mathcal{C}|& \leq& ((1-1/e)^{2} + o(1))(n-t)!/2 + (\log n) (n-t-1)!/2 \\
&=& ((1-1/e)^{2}+o(1))(n-t)!/2
\end{eqnarray*}
contradicting (\ref{eq:alternatingboundC}), provided \(n\) is sufficiently large. This proves the claim.

Since we are assuming \(|\mathcal{A}| \geq |\mathcal{E}| + t\), we must have equality, so \(\mathcal{A}=\mathcal{B}\), proving Theorem \ref{thm:hiltonmilnertypeAn}.
\end{proof}


\end{document}